%% file: eng.tex
\documentclass[8pt]{article}
\usepackage[margin=1.2cm]{geometry}

\usepackage{babel}

\input{header}
\DeclareMathOperator{\Rre}{\operatorname{Re}}
\DeclareMathOperator{\Iim}{\operatorname{Im}}
\DeclareMathOperator{\sppan}{\operatorname{span}}
\DeclareMathOperator{\disc}{\operatorname{disc}}
\DeclareMathOperator{\dist}{\operatorname{dist}}

\begin{document}

\title{Lov{\'a}sz theta approach to eventown problem}

\author{Mikhaylo Antipov\footnote{National Research University Higher School of Economics, Soyuza Pechatnikov str., 16, St. Petersburg, Russian Federation.}~~and Danila Cherkashin\footnote{St. Petersburg Department
of Steklov Mathematical Institute
of Russian Academy of Sciences;
27 Fontanka, St. Petersburg, Russia.}}

\maketitle

\begin{abstract}
We apply the bound on independence number via Lov{\'a}sz theta function to eventown problem and its generalizations over $\mathbb{Z}_n$.    
\end{abstract}



\section{Introduction}

The following bound was introduced by Lov{\'a}sz in his celebrated paper on Shannon capacity. In particular it implies famous Hoffman bound~\cite{haemers2021hoffman}. We provide the proof in the interest of completeness.

\begin{theorem}[Lov{\'a}sz,~\cite{lovasz1979shannon}]
Let $G = (V,E)$ be a simple graph. Consider a symmetric real matrix $A$ such that $A_{ij} = 1$ for every pair $\{i,j\} \notin E(G)$.
Then
\[
\alpha(G) \leq \lambda_{max}(A),
\]
where $\alpha(G)$ is the size of a maximal independent set of $G$ and $\lambda_{max}(A)$ is the maximal eigenvalue of $A$.
\label{lovasz}
\end{theorem}

\begin{proof} Let $I$ be an independent set, denote by $\chi_I$ its characteristic vector. Then
\[
(A\chi_I, \chi_I) = |I|^2.
\]
From the other hand, one has
\[
(A\chi_I, \chi_I) = \sum c_i^2\lambda_i \leq \left(\sum c_i^2 \right) \lambda_{max} = |I| \cdot \lambda_{max},
\]
where $\chi_I = \sum c_i v_i$ is the decomposition of $\chi_I$ via orthonormal eigenbasis $\{v_i\}$ of $A$. 
(We use that symmetric matrix has real spectrum and the length of $\chi_I$ is the same in the standard basis and in $\{v_i\}$, i.e. $\sum c_i^2 = |I|$.)
\end{proof}

The minimum of $\lambda_{max}(A)$ over the appropriate $A$ is called \textit{Lov{\'a}sz number} or \textit{Lov{\'a}sz theta function} of a graph.

Also we need the following corollaries. 
Suppose that $A$ and $G$ satisfy the conditions of Theorem~\ref{lovasz}.
Let $c$, $\lambda_{max}$ and $\spp$ stand for the minimal entry, the maximal eigenvalue and the spectral radius of $A$, respectively.

\begin{corollary}
Let $I$ be a set with at most $\varepsilon |I|^2/2$ edges inside. 
Suppose that $(1-c)\varepsilon < 1$. Then
\[
|I| \leq \frac{\lambda_{max}}{1-(1-c)\varepsilon}.
\]
\label{supersaturation}
\end{corollary}

\begin{corollary}
Let $I$ and $J$ be subsets of $V(G)$ with at most $\varepsilon |I| \cdot |J|$ edges between $I$ and $J$ (edges in $I\cap J$ are counted twice here).
Suppose that $(1-c)\varepsilon < 1$. Then
\[
|I| \cdot |J| \leq \left( \frac{\spp}{1-(1-c)\varepsilon} \right)^2.
\]
\label{cross}
\end{corollary}

For the bounds on \textit{disjoint} $I$ and $J$ one can make the class of appropriate matrices slightly wider, i.e.
not demand $A_{ii} = 1$. Then one may combine the proof of Corollary~\ref{cross} with Proposition 4.1 in the paper of Haemers~\cite{haemers2001bicliques}.

\subsection{A straightforward application to eventown problem}

Let $F$ be a family of subsets of $[n]$ is \textit{eventown} if the intersection of any two members is even (in particular all sets have even size).
Berlekamp~\cite{berlekamp1969subsets} and Graver~\cite{graver1975boolean} independently proved $F$ has at most $2^{\lf n/2\rf}$ members, which is also best possible.
The proof is very short up to general linear algebra. Note that every maximal eventown $F$ is a linear subspace of $\mathbb{F}_2^n$; otherwise one can replace $F$ with $\sppan F$.
Since $F$ lies in the orthogonal complement $F^\perp$ and $\dim F + \dim F^\perp = n$, $|F|$ has the dimension at most $\lf n/2 \rf$.

Consider the following Hadamard matrix
\[
A = 
\begin{pmatrix}
1 & 1 \\
1 & -1
\end{pmatrix}.
\]
Its spectrum is $\{\pm \sqrt{2}\}$. 
Then the spectrum of $M=A^{\otimes n}$ is $\{\pm 2^\frac{n}{2} \}$. 
Let us consider $G=(2^{[n]},E)$ and $(X,Y)\in E$ iff $|X\cap Y|$ is odd. Then we identify elements of $2^{[n]}$ with $\{0,1\}^n$ by usual way as well as indices for rows and columns of matrix $M$ (where we mean that  $A=(a_{ij})_{i,j\in \{0,1\}}$). Let $M=(m_{st})_{s,t\in \{0,1\}^n}$. Then for each $X,Y\in 2^{[n]}$ we have
\[
 m_{\chi(X),\chi(Y)}=\prod_{r=1}^n a_{(\chi(X))_r, (\chi(Y))_r}=(-1)^{|\{ r\in \{1,\dots,n\}| a_{(\chi(X))_r, (\chi(Y))_r}=-1 \}|}=(-1)^{|X\cap Y|}.
\]
Thus we see that graph $G$ and matrix $M$ satisfy the conditions of Theorem~\ref{lovasz}.

Applying Theorem~\ref{lovasz} one has $|F| \leq 2^{n/2}$. For even $n$ we already get another proof of eventown theorem.
For even $n$ we should also recall that $F$ is a linear subspace, so $|F| \leq 2^{\lf n/2 \rf}$.

Let $op(F)$ denote the number of distinct pairs $f_1, f_2 \in F$ for which $|f_1 \cap f_2|$ is odd.
O'Neill~\cite{oNeill2021towards} showed that for $1 \leq s \leq 2^{\lf n/2 \rf} - 2^{\lf n/4 \rf}$ there is a family $F$ with $|F| = 2^{\lf n/2 \rf} + s$
and $op(F) = s \cdot 2^{\lf n/2 \rf - 1}$.
Also he conjectured that this example is tight and proved the conjecture for $s = 1,2$.
The application of Corollary~\ref{supersaturation} gives twice weaker bound for even $n$ (and much weaker bound for odd $n$).

\begin{theorem}
\label{theoremop}
Let $|F| = 2^{n/2} + s$ for some integer $s$. Then
\[
op(F) \geq s \cdot 2^{\lf \frac{n}{2} \rf - 2}.
\]
\end{theorem}

\subsection{An application to $k$-town problem}

Let $F$ be a family of vectors from $\{0,\dots,k-1\}^n$, such that $(f_1,f_2) = 0 \pmod k$ for any $f_1,f_2 \in F$ (in particular for $f_1 = f_2$);
such $F$ is further called a \textit{$k$-town family}.

First, for prime $k$ the classical argument gives the tight upper bound $k^{\lf n /2\rf}$.

If $k$ is square free we can obtain the same bound from the prime $k$ case in the following way. Let us see that the case $k=pq$ is a formal consequence of the cases $k=p$, $k=q$ if $p,q$ are coprime. Indeed let $F\subset (\mathbb{Z}/pq\mathbb{Z})^n$ such that $(f_1,f_2) = 0$ for each $f_1,f_2 \in F$.  Then by assumption the number of residues modulo $p$ for elements $F$ does not exceed $p^{\frac{n}{2}}$ and number of residues modulo $q$ for elements $F$ does not exceed $q^{\frac{n}{2}}$. Then by the Chinese remainder theorem  $|F|\leq p^{\frac{n}{2}}q^{\frac{n}{2}}$.

The observations above should be folklore, meanwhile we do not know how to prove the first inequality in the following theorem without spectral graph theory for an arbitrary $k$. 
\begin{theorem}
If $F$ is a $k$-town family then 
\[
|F|\leq k^{\frac{n}{2}}.
\]
Moreover suppose that $k$ is prime, and $(f_1,f_2) \neq 0 \pmod k$ for at most $\varepsilon|F|^2$ pairs $f_1,f_2 \in F$.
If $\varepsilon < \frac{k-1}{k}$ then
\[
|F| \leq \frac{k^{\frac{n}{2}}} {1 - \frac{k}{k-1} \varepsilon}.
\]
\label{ktown}
\end{theorem}

Obtaining an example of $k$-town $F$ with $|F| = k^{\frac{n}{2}}$ in the case of prime $k$ and even $n$ is equivalent to finding a set of $\frac{n}{2}$ pairwise (and-self) orthogonal linear independent vectors in $(\mathbb{Z}/k\mathbb{Z})^n$. There are a lots of such sets; for example for a prime $k=4t+1$ one can consider vectors of the form $v_j=e_{2j-1}+\varepsilon^t e_{2j}$ for $1\leq j\leq \frac{n}{2}$ (here $\{e_j\}$ --- is the standard basis in $(\mathbb{Z}/k\mathbb{Z})^n$, $\varepsilon$ is a primitive root in $\mathbb{Z}/k\mathbb{Z}$). 

For a general prime $k$ one can choose $v_1,v_2,\dots$ inductively and almost arbitrarily such that $v_j\in \langle v_1,\dots,v_{j-1}\rangle^{\perp}\setminus \langle v_1,\dots,v_{j-1}\rangle$  and $(v_j,v_j)=0$. This can be done: indeed, for $j<\frac{n}{2}-2$ we can choose arbitrary $4<\dim \langle v_1,\dots,v_{j-1}\rangle^{\perp} - \dim \langle v_1,\dots,v_{j-1}\rangle$ linearly independent vectors and find $v_j$ in their span (as any quadratic form with $\geq 3$ variables over finite field has an isotropic vector). When $j=[\frac{n}{2}]-2$, we can choose $4$-dimensional 
$V\subset (\langle v_1,\dots,v_{j-1}\rangle^{\perp}\setminus \langle v_1,\dots,v_{j-1}\rangle)\cup{0}$ as well, 
and it is also well-known that there exists $v_{\frac{n}{2}-1}, v_{\frac{n}{2}}\in V$ such that $(v_{\frac{n}{2}-1},v_{\frac{n}{2}-1})=(v_{\frac{n}{2}-1},v_{\frac{n}{2}})=
(v_{\frac{n}{2}},v_{\frac{n}{2}}) = 0$.

In some cases were $k$ is non-prime  we can obtain examples of different nature. For example, when $k=m^{2}$ for some integer $m$ one can consider the set of vectors of the form $(mx_1,\dots, mx_n)$. This example shows that in the case of $k$ being a perfect square the first inequality of Theorem~\ref{ktown} is also tight for an odd $n$.   

Suppose that we are interested in the scalar product $t$ instead of $0$. Then the statement of Theorem~\ref{ktown} can be slightly improved.

\begin{corollary}
Let $F$ be a family of vectors from $\{0,\dots,k-1\}^n$, and $(f_1,f_2) = t \pmod k$ for every $f_1,f_2 \in F$. Then
\[
|F| \leq c(t,k) \cdot k^{\frac{n}{2}}
\]
for some $\frac{1}{\sqrt{2}} < c(t,k) \leq 1$. Moreover if $\frac{k}{\gcd (k,t)}$ tends to infinity $c(t,k)$ tends to $\frac{1}{\sqrt{2}}$ ($\gcd$ stands for the greatest common divisor).

Assume also that $k$ is prime and $(f_1,f_2) \neq t \pmod k$ for at most $\varepsilon|F|^2$ pairs $f_1,f_2 \in F$ for some $t \neq 0$.
If $\varepsilon < \frac{k-1}{k}$ then
\[
|F| \leq c(k) \frac{k^{\frac{n}{2}}} {1 - \frac{k}{k-1} \varepsilon},
\]
where $c(k) < 1$ and $c(k) \to \frac{2\sqrt{2}}{\pi}$ with $k \to \infty$.
\label{slightlybetter}
\end{corollary}

\section{Proofs}

\begin{proof}[Proof of Corollary~\ref{supersaturation}]
Denote by $\chi_I$ the characteristic vector of $I$. Then
\[
(A\chi_I, \chi_I) \geq |I|^2 \cdot ( 1-(1-c)\varepsilon ).
\]
From the other hand, one has
\[
(A\chi_I, \chi_I) = \sum c_i^2\lambda_i \leq \left(\sum c_i^2 \right) \lambda_{max} = |I| \cdot \lambda_{max},
\]
where $\chi_I = \sum c_i v_i$ is the decomposition of $\chi_I$ via orthonormal eigenbasis $\{v_i\}$ of $A$. 
\end{proof}

\begin{proof}[Proof of Corollary~\ref{cross}]
Denote by $\chi_I$ and $\chi_J$ the characteristic vectors of $I$ and $J$ respectively. Then
\[
(A\chi_I, \chi_J) \geq |I| \cdot |J| \cdot ( 1-(1-c)\varepsilon ).
\]
From the other hand, one has
\[
(A\chi_I, \chi_I) = \sum c_i d_i \lambda_i \leq \left (\sum |c_i| \cdot |d_i| \right ) \spp \leq \sqrt{\left(\sum c_i^2 \right)  \left(\sum d_i^2 \right) } \cdot \spp = 
\sqrt{|I| \cdot |J|} \cdot \spp,
\]
where $\chi_I = \sum c_i v_i$ and $\chi_J = \sum d_i v_i$ are the decompositions of $\chi_I$ and $\chi_J$ via orthonormal eigenbasis $\{v_i\}$ of $A$. 

\end{proof}

\begin{proof}[Proof of Theorem~\ref{ktown}]

Consider the following $k\times k$ matrix:
\[
a_{jl} = \phi^{jl},
\]
where $\phi$ is a primitive $k$-th root of unity and $0 \leq j,l \leq k-1$ and the matrix $M=A^{\otimes n}$. Then we have $A^4=k^{2}E$, therefore $|\lambda|=(k^{\frac {1}{2}})^n=k^\frac{n}{2}$ for each eigenvalue $\lambda$ of $M$, i.e., $\spp(M)=k^\frac{n}{2}$. 

Let us see that $A$ (and, as a consequence, $M$) has an eigenbasis in $\mathbb{R}^n$. Let $\{e_i\}$ be a standard basis. Then moving to the basis
$\{e_1, \frac{1}{\sqrt{2}}(e_j\pm e_{k+2-j})\}$ or to $\{e_1, e_{\frac{k}{2}}, \frac{1}{\sqrt{2}}(e_j\pm e_{k+2-j})\}$ for even $k$ (note that it is a unitary transformation) we obtain a block matrix with two blocks of the form $A_1, iA_2$, where $A_1,A_2$ are the real symmetric matrices (they have the sizes $\frac{k+1}{2}$ и $\frac{k-1}{2}$ respectively for odd $k$ and $\frac{k}{2}+1$, $\frac{k}{2}+1$ for even $k$). As a matter of fact, for each $2\leq j\leq k$ we have
\[
A(e_j+e_{k+2-j})=\sum_{l=1}^k \phi^{(j-1)(l-1)} e_l +
\sum_{l=1}^k \phi^{(k+1-j)(l-1)} e_l = 2+ 
\sum_{l=2}^k (\phi^{(j-1)(l-1)}  +
  \overline{\phi^{(j-1)(l-1)}  })e_l = 2+
  \sum_{l=2}^{[\frac{k}{2}]}2\cdot \Rre (\phi^{(j-1)(l-1)})(e_l+e_{k+2-l})
\]
and for each $2\leq j\leq k$, $j\neq \frac{k}{2}$
\[
A(e_j-e_{k+2-j})=\sum_{l=1}^k \phi^{(j-1)(l-1)} e_l -
\sum_{l=1}^k \phi^{(k+1-j)(l-1)} e_l =  
\sum_{l=2}^k (\phi^{(j-1)(l-1)}  -
  \overline{\phi^{(j-1)(l-1)}  })e_l =
  i\cdot \sum_{l=2}^{[\frac{k}{2}]}2\cdot \Iim(\phi^{(j-1)(l-1)})(e_l-e_{k+2-l}).
\]

Thus we see that this change of the base (over $\mathbb{R}$) leads to real symmetric matrix and to pure imaginary symmetric matrix, which both have real eigenbasis, therefore $A$ has real eigenbasis as well. Hence $N := \Rre M$ has the same real eigenbasis with $M$. 
Obviously, all eigenvalues of $N$ lie in $\{\pm k^{n/2},0\}$.

Let us consider $G=((\mathbb{Z}/k\mathbb{Z})^n,E)$, where  $(X,Y)=((x_1,\dots,x_n), (y_1, \dots, y_n))\in E$ iff $\sum_{ r=1}^n x_ry_r\neq 0$. Then we also identify $(\mathbb{Z}/k\mathbb{Z})^n$ as indices for rows and columns of matrix $M$ (taking all indices in $A$ modulo $k$). Let $M=(m_{st})_{s,t\in \{0,\dots, k-1\}^n}$. Now for each $X,Y\in ((\mathbb{Z}/k\mathbb{Z})^n, (X,Y)\notin E$ we have
\[
 m_{X,Y}=\prod_{r=1}^n a_{x_r, y_r}=\phi^{\sum_{r=1}^n x_ry_r}=\phi^{0}=1.
\]
Thus graph $G$ and matrix $M$ (and also $N$) satisfy the conditions of Theorem~\ref{lovasz}, so the first statement of the theorem is proved.

Now let $k$ be a prime number. An immediate application of Corollary~\ref{supersaturation} gives
\[
|F| \leq \frac{k^{\frac{n}{2}}} {1 - t \varepsilon},
\]
where $t = 1 - \cos \left ( \lf \frac{k}{2} \rf \frac{2 \pi}{k} \right )$, which tends to 2 with $k \to \infty$.
So we modify the proof of Corollary~\ref{supersaturation} in the following way.

For every root $\phi$ the matrix $N = N(\phi)$ satisfies
\[
(N(\phi) \chi_F, \chi_F) = \sum_{i,j \in F} N_{ij}(\phi) 
\]
and
\[
(N(\phi) \chi_F, \chi_F) \leq |F| \cdot k^{\frac{n}{2}}.
\]
Summing up these inequalities for all $k$-th roots except $1$ one has
\begin{equation}
 (k-1) \cdot |F|^2 \cdot (1-\varepsilon) - \varepsilon \cdot |F|^2  \leq (k-1) \cdot |F| \cdot k^{\frac{n}{2}}
\label{13231}    
\end{equation}
since for every $i,j$ that corresponds to sets with nonzero scalar product
\[
\sum_{\phi \neq 1} N_{ij}(\phi) = -1,
\]
and for $i,j$ that corresponds to sets with zero scalar product
\[
\sum_{\phi \neq 1} N_{ij}(\phi) = k-1.
\]
Rewriting~\eqref{13231} finishes the proof.
\end{proof}

\begin{proof}[Proof of Corollary~\ref{slightlybetter}]
Fix a primitive $k$-th root of unity $\phi$ and consider the same matrices $A$ and $M$ as in the previous proof. 
Let us consider $G=((\mathbb{Z}/k\mathbb{Z})^n,E)$, where  $(X,Y)=((x_1,\dots,x_n), (y_1, \dots, y_n))\in E$ iff $\sum_{ r=1}^n x_ry_r\neq t$.
Then $G$ and $N := \Rre (\phi^t M)$ satisfy the conditions of Theorem~\ref{lovasz}.
Note that $N$ and $M$ shares a real eigenbasis, so all eigenvalues of $N$ lie in $\Rre \{ \pm \phi^t k^{n/2}, \pm \phi^t k^{n/2}i \}$. 
Hence the spectral radius of $N$ lies between $\frac{1}{\sqrt{2}} k^{n/2}$ and $k^{n/2}$. 

Lets, check the second part of first proposition. Note that for any coprime  $x,y$ we have $\frac{xy}{\gcd (xy,t)}=\frac{x}{\gcd (x,t)}\cdot \frac{y}{\gcd (y,t)}$ and $c(t,xy)\leq c(t,x)c(t,y)$ --- this follows from remarks before Theorem~\ref{ktown}. Therefore we need only treat the case $k=p^s$ for prime $p$.

Note that any upper bound for $|F|$ with some $t=t_0$ is also an upper bound for $t=t_0r^2, r\in (\mathbb{Z}/k\mathbb{Z})^*$. Indeed, if $(f_1,f_2)=t_0r^2$ for each $f_1,f_2\in F$, then $(\frac{1}{r}f_1, \frac{1}{r}f_2)=t_0$ and we can apply a bound for $t=t_0$. Now we are to show that we can choose $r=r(k,t)$ such that $|\Rre (\phi^{tr^2})|$ tends to $\frac{1}{\sqrt{2}}$ if $\frac{k}{\gcd (k,t)}$ tends to infinity --- in this case $|\Rre (\phi^{tr^2}i^l)|$ tends to $\frac{1}{\sqrt{2}}$ as well. 
To obtain this, it's sufficient to choose $r$ such that $|\frac{tr^2}{k/8}|$ tends to $1$ (with suitable choice of  representative of $tr^2$ modulo $k$).

Any $t\in \mathbb{Z}/p^s \mathbb{Z}$ we can consider as $t=p^mt', r\in\mathbb{Z}/p^{s-m}\mathbb{Z}$ and replacing $(t,p^s)$ with $(r, p^{s-m})$ (this doesn't change both ratios in question) we can assume that $t\in (\mathbb{Z}/p^s \mathbb{Z})^*$ (and $\frac{k}{\gcd (k,t)}=k$).

Let $p>2$. Let's prove that we can choose $l$ such that $|l|<k^{\frac{1}{2}}$ and $t=lw^2$ modulo $k$ for some $w$. Suppose that $t$ is non-square modulo $k$ (else we can take $l=1$). If $p=2$, then all non-squares are known to be of the form
$3w^2,5w^2$ or $7w^2$ for some $w$, so we can take $l\in\{1,3,5,7\}$ (when $s>5$). If $p>2$ we can take $l_1,l_2$ with $|l_i|<\sqrt{p}$ such that $l_1=l_2t$ modulo $p$ by Thues lemma. Then for some $i$ $l_i$ is a non-square modulo $m$ and therefore modulo $k=p^s$. As $t$ is non square as well we can write $t=l_iw^2$ and take $l=l_i$.


Now take $x \in \mathbb{Z}$ such that $x^2\in \left (\frac{k}{8|l|} - \left( \frac{k}{|l|} \right)^{\frac{1}{2}}, \frac{k}{8|l|}+ \left (\frac{k}{|l|} \right)^{\frac{1}{2}} \right)$, then
$||lx^2|-\frac{k}{8}| \leq (kl)^{\frac{1}{2}}<k^{\frac{3}{4}}$
and $|\frac{lx^2}{\frac{k}{8}}|$ tends to $1$ when $k$ tends to infinity, as desired (we put $r=\frac{x}{w}$ modulo $k$).

Now we prove the second part of Corollary. Repeat the proof of the corresponding part of Theorem~\ref{ktown}.
Then the right part of~\eqref{13231} is replaced with
\[
\sum_{a=1}^{k-1} \spp  \Rre (\phi^{ta} M) .
\]
Note that for every $a$ 
\[
\spp  \Rre (\phi^{ta} M) = \max ( \Rre (\phi^{ta}) , \Rre(\phi^{ta}\cdot i) , -\Rre(\phi^{ta}) , -\Rre(\phi^{ta}\cdot i) ) \cdot k^{n/2} =
\]
\[\max \left ( \left | \cos \frac{2\pi t a }{k} \right |,\left | \cos  \left ( \frac{2\pi t a }{k} + \frac{\pi}{2} \right ) \right |, \left | \cos  \left ( \frac{2\pi t a }{k} + \pi \right )\right |, \left | \cos  \left ( \frac{2\pi t a }{k} + \frac{3\pi}{2} \right ) \right | \right ) k^{n/2}.
\]
After summation over $1 \leq a \leq k-1$ and taking $k^{n/2}$ out of the brackets and dividing by $k-1$, one has a Riemann sum for the following integral
\[
\frac{\int_{0}^{\pi/4} \cos(x) d x }{\pi/4} = \frac{2\sqrt{2}}{\pi}.
\]

\end{proof}

\section{Discussion}

\subsection{Related subjects}

It turns out that Theorem~\ref{ktown} is closely related with several well-studied themes.

\paragraph{Erd{\H o}s distance problem over a finite field.}
Let $\mathbb{F}_q$ be a finite field. Hart and Iosevich~\cite{hart2008sums} proved that if $t \neq 0$ and $E \subset \mathbb{F}^n_q$, then
\[
|\{(x, y) \in E \times E :  x \cdot y = t\}| = \frac{|E|^2}{q} + R(E),
\]
where
\[
|R(E)| \leq q^\frac{n-1}{2} |E|.
\]
In the notation of Theorem~\ref{ktown} (for this part we assume $k$ to be prime) it implies inequality
\[
|F| \leq \frac{k^\frac{n-1}{2}}{\frac{k-1}{k} - \varepsilon}.
\]
This bound is better than Theorem~\ref{ktown} for a large prime $k$, but it essentially requires the condition $t \neq 0$. The bound for $t = 0$ is obtained in~\cite{hart2011averages} and is slightly weaker than Theorem~\ref{ktown}.

Define the distance in an $n$-dimensional vector space over a finite field as
\[
\dist(\bar x, \bar y) = (x_1 - y_1)^2 + \dots (x_n - y_n)^2,
\]
i.e. we omit the square root in the standard formula.
Let $\Delta(E)$ be the set of all distances between points in $E$.
The Erd{\H o}s distance problem is to determine the smallest possible cardinality of the set $\Delta(E)$ over all sets $E$ of a given size.

In the case of prime $k > 2$ Theorem~\ref{ktown} gives the following elementary bound which is far from the Iosevich--Rudnev estimate~\cite{iosevich2007erdos}.
Let $E$ be a set with $|\Delta(E)| = s$. Without loss of generality $(0,\dots,0) \in E$; otherwise one may shift $E$. 
Then $E$ lies on $s$ spheres centered at $(0,\dots,0)$. 
By a pigeon-hole principle there is a sphere containing a subset $E'\subset E$ of size at least $(|E|-1)/s$. 
Since $k$ is odd and $E'$ lie on a sphere, there are at most $s$ different scalar products between vectors of $E'$.
Let $t$ be a most popular scalar product in $E' \times E'$. Then one may apply Theorem~\ref{ktown} with $\varepsilon = \frac{s-1}{s}$ and receive that
\[
|E'| \leq \frac{k^{n/2}}{1 - \frac{k}{k-1}\frac{s-1}{s}} = \frac{s(k-1)k^{n/2}}{k-s}.
\]
Hence
\[
|E| \leq \frac{s^2(k-1)k^{n/2}}{k-s} + 1.
\]

For instance, $s = k-1$ means that we evaluate the size of a set without a single distance $r$ by $k^{{\frac{n+6}{2}}}$, meanwhile Iosevich--Rudnev estimate is
$Ck^{\frac{n+1}{2}}$ with an absolute constant $C > 0$.

\paragraph{A bound via singular numbers.}
Let $G = (V,E)$ be a simple graph. Suppose that $A$ is a complex matrix such that $A_{ij} = 1$ for every pair $\{i,j\} \notin E(G)$.
Then
\[
\alpha(G) \leq \sigma_{max}(A),
\]
where $\alpha(G)$ is the size of a maximal independent set of $G$ and $\sigma_{max}(A)$ is the maximal singular value of $A$ (i.e. the square root of a
maximal eigenvalue of a self-adjoint operator $A^*A$, where $A^*$ denotes the adjoint of $A$).
The proof immediately follows from the main theorem of~\cite{danciger2006min}.

For matrix $A$ in the proof of Theorem~\ref{ktown} one has
\[
A^*A = kE,
\]
so every singular number of matrix $A$ is equal to $\sqrt{k}$.
Hence every singular number of $M = A^{\otimes n}$ is equal to $k^{n/2}$.
This implies the first inequality of Theorem~\ref{ktown}.

\paragraph{Hypergraph discrepancy and asymptotic precision of Theorem~\ref{theoremop} when $\varepsilon$ is close to $1/2$.}
A hypergraph is a pair $(V, E)$, where $V$ is a finite set whose elements are called vertices and $E$ is a family of subsets of $V$, called edges. 
A vertex $2$-coloring of a hypergraph $(V, E)$ is a map $\pi : V \rightarrow \{\pm 1\}$. 
A \textit{discrepancy} of a coloring $\pi$ is the largest value of $|\sum_{v\in e} \pi (v) |$ over $e \in E$.

Consider an explicit hypergraph $H = (V,E)$, where $V = [N] \times [N]$ and edges have form $I \times J$ for $I,J \subset [N]$. 
Note that every $\{\pm 1\}$ matrix of size $N$ produces a 2-coloring of $H$.
Let $N = 2^n$ and $\chi$ be the coloring from $A^{\otimes n}$ (recall that $A$ is $2\times2$ Hadamard matrix).
Let $I \times J$ be an edge of $H$ providing the discrepancy $\disc$ of $\chi$; without loss of generality let $\disc$ be positive. Then
\[
\disc = (A^{\otimes n} \chi_I, \chi_J) =  (1-2\varepsilon) \cdot |I| \cdot |J|,
\]
where $\varepsilon$ satisfies the conditions of Theorem~\ref{theoremop}. Theorem~\ref{theoremop} implies 
\begin{equation}
    \disc^2 = |I|^2 \cdot |J|^2 (1-2\varepsilon)^2  \leq 2^n |I|\cdot |J| \leq 2^{3n} = N^3.
\label{precise}
\end{equation}

Astashkin proved~\cite{astashkin2010rademacher} that $H$ has the discrepancy at least $cN^{3/2}$ for \textit{every} coloring $\chi$ and some absolute constant $c > 0$, i.e. inequality~(\ref{precise}) is precise up to an absolute constant. It means that Theorem~\ref{theoremop} is precise up to an absolute constant in the case when $|I|$ and $|J|$ are close to $N$ and $\frac{1}{2} - \varepsilon$ is of order $2^{-n/2}$.

\subsection{Further questions}

\paragraph{Generalization on $k$-eventown.}
We say that $F \subset 2^{[n]}$ is $k$-eventown if the size of the intersection of any (not necessarily different) $f_1, f_2 \in F$ is zero modulo $k$.
The problem of determine the maximal size of $k$-eventown was studied by Frankl and Odlyzko~\cite{frankl1983subsets}.
They found a nice construction based on Hadamard matrices of $k$-eventowns of size at least $(ck)^{\lf n/(4k) \rf}$, where $c > 0$ is an absolute constant. 
In addition, they showed that any $k$-eventown has size at most $2^{O(\log k/k)n}$ as $n$ tends to $\infty$. 

In particular, for $k = 3$ the best known lower and upper bounds are $24^{\lf n/12 \rf}$ and $2^{\lf n/2 \rf}$ respectively.

\paragraph{Generalization on $t$-wise $k$-eventown.}
We say that $F \subset 2^{[n]}$ is $t$-wise $k$-eventown if the size of the intersection of any different $f_1,\dots, f_t \in F$ is zero modulo $k$.
Note that a 2-wise $k$-eventown is not the same as an $k$-eventown, since in the former we do not require that the sets themselves
have size zero modulo $k$. 

Sudakov and Vieira~\cite{sudakov2018two} show that a $t$-wise eventown has for $t \geq 3$ a unique extremal configuration and obtain a stability result for this problem. Gishboliner, Sudakov and Tomon~\cite{gishboliner2021small} show that for every $k$ there is $t = t(k)$ that
the size of any $t$-wise $k$-eventown is bounded by
\[
|F| \leq 2^{[n/k]} + const(k,t).
\]

A generalization of the methods of the paper, if exists, will have a deal with some tensor analysis.

\paragraph{Acknowledgments.} The authors are grateful to Fedor Petrov, Andrey Kupavskii and Pavel Prozorov for useful discussions. 
The reviewers remarks were extremely helpful. The work of Danila Cherkashin was supported by the Russian Science Foundation grant 21-11-00040.

\bibliographystyle{plain}
\bibliography{main.bib}

\end{document}

%% file: header.tex
\usepackage[utf8]{inputenc}
\usepackage[T2A]{fontenc}

\usepackage{intcalc, calc}
\usepackage[nomessages]{fp}
\usepackage{graphicx}
\usepackage{wrapfig}
\usepackage{float}

\def\lf{\left\lfloor}   
\def\rf{\right\rfloor}

\usepackage{tikz}
\usepackage{color}
\usepackage{framed}
\usepackage{ragged2e}
\usepackage{amsmath, amssymb, amsfonts, amsthm}
\usepackage{wasysym}
\usepackage{soul, soulutf8}
\usepackage{calc}
\usepackage{mathtools,mathcomp}
\usepackage{pifont}
\usepackage{enumitem}

\usepackage{ifthen}

\DeclareMathOperator{\spp}{\operatorname{sp}}

\def\divdots{\rlap{\raisebox{-1pt}{.}}{\rlap{\raisebox{2pt}{.}}\raisebox{5pt}{.}}}

\def\ndivby{\mathrel{
    \divdots
    \kern-0.35em\raise0.22ex\hbox{/}
}}

\renewcommand{\leq}{\leqslant}
\renewcommand{\geq}{\geqslant}

\def\fs{\kern 0.5em}
\newcounter{prcnt}
\newcounter{pucnt}
\newcommand{\prmain}[1]{ 
    \medskip%
    \setcounter{pucnt}{0}%
    \stepcounter{prcnt}%
    \noindent\textbf{%
    \theprcnt%
    \ifthenelse{\equal{#1}{1}}{*}{}%
    .}\fs%
}

\newcommand{\pumain}[1]{
    \stepcounter{pucnt}{%
    \noindent\bf(\alph{pucnt}%
    \ifthenelse{\equal{#1}{1}}{*}{}%
    )}\fs%
}

\newtheorem{theorem}{Theorem}

\newtheorem*{theorem*}{Теорема}
\newtheorem{corollary}{Corollary}
\theoremstyle{definition}

